\documentclass[11pt,a4paper,twoside]{article}

\usepackage[top=1in, bottom=1in, left=1.25in, right=1.25in]{geometry}

\usepackage[utf8]{inputenc}
\usepackage[english]{babel}

\usepackage{amsthm}
\usepackage{amssymb}
\usepackage{amsmath}

\usepackage{enumitem}

\usepackage{graphics}

\newtheorem{theoremalpha}{Theorem}

\newtheorem{theorem}{Theorem}[section]
\newtheorem*{theorem*}{Theorem}
\newtheorem{proposition}[theorem]{Proposition}
\newtheorem{lemma}[theorem]{Lemma}

\DeclareMathOperator{\RR}{\mathbb{R}}
\DeclareMathOperator{\vol}{Vol}
\DeclareMathOperator{\rank}{rank}

\newcommand{\modul}[1]{\left|#1\right|}
\newcommand{\norm}[1]{\left\lVert#1\right\rVert}

\newcommand{\dd}{\mathrm{d}}

\title{Equidistribution of horospheres in nonpositive curvature}
\author{Sergi Burniol Clotet\\
LPSM, Sorbonne Université, 4 Place Jussieu, 75005 Paris, France\\
(email: sergi.burniol\_clotet@upmc.fr)}
\date{June 23, 2020}

\begin{document}

\setlength{\parskip}{0.5ex plus 0.5ex minus 0.2ex}
\maketitle

\begin{abstract}
	We study the ergodic properties of horospheres on rank $1$ manifolds with nonpositive curvature. We prove that the horospheres are equidistributed under the action of the geodesic flow towards the Bowen-Margulis measure, on a large class of manifolds.
	In the case of surfaces, we define a parametrization of the horocyclic flow on the set of horocycles containing a rank $1$ vector that is recurrent under the action of the geodesic flow. We prove that the horocyclic flow in restriction to this set is uniquely ergodic.
	The results are valid for large classes of manifolds, including the compact ones.
\end{abstract}



\section{Introduction}
Horocyclic flows associated to a geodesic flow have been extensively studied on compact surfaces with constant negative curvature \cite{Furstenberg1}, and later on compact surfaces with variable negative curvature \cite{Marcus1,Marcus77}. They are uniquely ergodic and mixing, and they have zero topological entropy, among other properties. More generally, on negatively curved compact manifolds of any dimension, the Bowen-Margulis measure is the unique measure invariant under the unstable foliation, and all the horospheres are equidistributed towards this measure \cite{Roblin03}. In this paper, we have two goals: first, we produce a result on the equidistribution of horospheres for rank $1$ manifolds with nonpositve curvature; and second, for the case of surfaces, we prove the unique ergodicity of the horocyclic flow restricted to a well-chosen subset of rank $1$ vectors.

M. Babillot gave a simple proof of the mixing property of the geodesic flow and showed the equidistribution of horospheres under the action of this flow towards certain product measures for manifolds with negative curvature \cite{Babillot1}. For the Bowen-Margulis measure, the equidistribution of horospheres can be stated as follows.

\begin{theorem*} \cite[Theorem 3]{Babillot1}
	Let $M$ be a non-elementary complete connected Riemannian manifold with negative curvature bounded away from $0$. Assume that the geodesic flow $g_t$ on the unitary tangent bundle $T^1M$ of $M$ is topologically mixing on the set of nonwandering vectors, and that the Bowen-Margulis measure $\mu$ is finite. Then, for every unstable horosphere $H\subset T^1M$, every open subset $U$ of $H$ containing a nonwandering vector is equidistributed under the action of the geodesic flow; i.e. for every bounded and uniformly continuous function $f$ on $T^1M$, we have
	\begin{equation*}
	\frac{1}{\mu_{H}(U)}\int_U {f}\circ g_t \, \dd \mu_{H} \xrightarrow[t \to +\infty]{} \frac{1}{\mu(T^1 M)}\int_{T^1 M}f \,\dd \mu, 
	\end{equation*}
	where $\{\mu_H\}_H$ are the conditional measures of the Bowen-Margulis measure $\mu$ along the unstable foliation.
\end{theorem*}

We refer the reader to \cite{PaulinPollicottSchapira} for the so-called Patterson-Sullivan construction of the Bowen-Margulis measure in negative curvature. Several criteria for the finiteness of this measure are given in \cite{PitSchapira}.

The Bowen-Margulis measure can be generalized to  nonpositively curved rank $1$ manifolds. G. Knieper constructed this measure in \cite{Knieper98}, following the method pioneered by Patterson and Sullivan \cite{Patterson76,Sullivan79}, and proved that it is the unique measure of maximal entropy when the manifold is compact. In this article, we explain an optimal way to generalize the equidistribution of horospheres towards Bowen-Margulis for nonpositively curved rank $1$ manifolds, following the approach of Babillot.

In the equidistribution theorem, we consider the averages of a function with respect to a measure $\mu_H$, which we will define in Section \ref{Sec:notation}, associated to the horocycle $H$. On a negatively curved manifold, if $U$ is an open subset of $H$ containing a nonwandering vector, the $\mu_H$-measure of $U$ is positive, so it makes perfect sense to average a function over $U$. However, on manifolds of nonpositive curvature, not every open subset of a horocycle containing a nonwandering vector has positive measure. As an example, we take a nonflat surface containing a flat cylinder (Figure \ref{fig:pic1}). All the vertical vectors with base point in a longitudinal segment of the cylinder are in the same unstable horocycle. The set $U$ formed by these vectors has zero $\mu_H$-measure, which is clear from its construction, although each vector of $U$ is periodic and, in particular, nonwandering. 

\begin{figure}[h]
	\centering
	\includegraphics{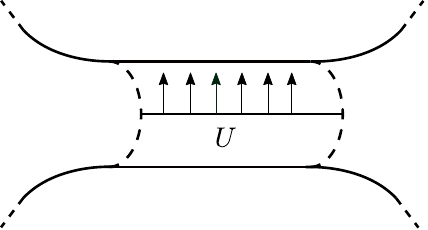}
	\caption{ \label{fig:pic1}
		A surface with a flat cylinder.}
\end{figure}

Theorem \ref{theorem_timeed} shows that, under certain hypothesis, an open subset $U$ of a horosphere $H$ is equidistributed in time, as soon as $U$ has positive $\mu_H$-measure. We emphasize that rank $1$ compact manifolds with nonpositive curvature satisfy the hypothesis, so there is equidistribution.

\begin{theoremalpha}\label{theorem_timeed}
	Let $M$ be a nonpositively curved non-elementary complete connected Riemannian manifold with a closed rank $1$ geodesic. Assume that the geodesic flow $g_t$ on the unitary tangent bundle $T^1M$ of $M$ is topologically mixing on the set of nonwandering vectors, and that the Bowen-Margulis measure $\mu$ is finite. Then, for every horosphere $H\subset T^1 M$ containing a nonwandering vector, every open subset $U$ of $H$ of finite but positive $\mu_H$-measure is equidistributed under the action of the geodesic flow; i.e. for every bounded and uniformly continuous function $f$ on $T^1M$, we have
	\begin{equation*}
	\frac{1}{\mu_{H}(U)}\int_U {f}\circ g_t \, \dd \mu_{H} \xrightarrow[t \to +\infty]{} \frac{1}{\mu(T^1 M)}\int_{T^1 M}f \,\dd \mu .
	\end{equation*}
\end{theoremalpha}

In the case of a piece of horocycle $U$ with zero $\mu_H$-measure, it could be reasonable to wonder if there is equidistribution with respect to another measure giving a positive value to $U$, for instance, the Lesbesgue measure, which is very natural. The example of the surface with a flat cylinder shows that this is also not possible, because the geodesic flow acts periodically on the piece of horocycle $U$; then, for a well-chosen function $f$, the averages of $f\circ g_t$ on $U$ could oscillate endlessly.

In the second part of the article, we study the case of surfaces, which is easier to deal with as the horospheres can be parametrized by a flow.  The dynamical properties of horocyclic flows are well understood in some situations. For instance, on a negatively curved compact surface, the Bowen-Margulis measure is the unique probability measure invariant under the horocyclic flow \cite{Marcus1}. For geometrically finite manifolds, there is a classification of the Radon measures invariant under the horocyclic flow \cite[Corollary 6.5]{Roblin03}.

We follow Marcus's method, based on the definition of a parametrization of the horocyclic flow by the measures on the horocycles. Unfortunately, it is not possible to define an analogous parametrization on the whole space for rank $1$ surfaces with nonpositive curvature, due to the presence of flat regions. We avoid this difficulty, by restricting our system to the set $\Sigma$ of vectors whose horocycle contains a rank $1$ vector recurrent under the geodesic flow. Under the hypothesis of the theorem, this set has full Bowen-Margulis measure and is $G_\delta$-dense in the unitary tangent bundle. Thanks to the equidistribution of the horocycles of Theorem \ref{theorem_timeed} and a part of the strategy followed by Y. Coudène in \cite{Coudene1}, we prove the unique ergodicity of the horocyclic flow on $\Sigma$ for manifolds that satisfy the duality condition, which means that the nonwandering set of the geodesic flow is the whole unitary tangent bundle. Rank $1$ compact surfaces with nonpositive curvature are included.

\begin{theoremalpha}\label{unique_ergodi}
	Let $M$ be an orientable rank $1$ complete connected Riemannian surface with nonpositive curvature satisfying the duality condition. Assume that the Bowen-Margulis measure $\mu$ is finite. Then every finite Borel measure on $\Sigma$ invariant under the horocyclic flow $h_s$ is a constant multiple of the Bowen-Margulis measure $\mu|_\Sigma$ restricted to $\Sigma$.
\end{theoremalpha}

In this article, we only work with expanding horospheres and expanding horocyclic flows, but all the results have an analogy in the contracting setting.

\subsection*{Acknowledgement}

The author wishes to thank his Ph.D. advisor Yves Coudène for the many helpful suggestions.

\noindent
\begin{minipage}[c]{0.08 \linewidth}
	\includegraphics{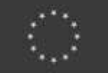}	
\end{minipage}
\begin{minipage}[c]{0.9\linewidth}
	This project has received funding from the European Union’s Horizon 2020 research and innovation program under the Marie Skłodowska-Curie grant agreement No. 754362.
\end{minipage}

\section{Measures on the horocycles and equidistribution}

\subsection{Notation}

Let $M$ be a complete connected Riemannian manifold of nonpositive curvature and denote by $g_t:T^1M\rightarrow T^1M$ the geodesic flow on the unitary tangent bundle $T^1M$ of $M$. Recall that the \emph{rank} of a vector $v$ in $T^1 M$,  denoted by $\rank v$, is the dimension of the parallel Jacobi fields along the geodesic tangent to $v$. There is always a parallel Jacobi field in the tangent direction of the geodesic, so the rank must be between $1$ and the dimension of $M$. We say that the manifold $M$ is of \emph{rank $1$} if it contains at least one vector of rank $1$. Let us start with some standard definitions and recall the most important facts about rank $1$ manifolds.

Most part of the reasoning takes place on the universal cover $\tilde{M}$ of $M$. We consider the Riemannian distance $d$ on $\tilde{M}$ and the distance $d_1$ associated to the Sasaki metric on $T^1\tilde{M}$. A \emph{geodesic ray} is, by definition, a map $\sigma : [0,+\infty)\rightarrow \tilde{M}$ that minimizes length. Two geodesic rays $\sigma_1,\sigma_2$ are \emph{asymptotic} if the distance $d(\sigma_1(t),\sigma_2(t))$ is uniformly bounded for all $t\ge0$. The \emph{boundary at infinity} $\partial\tilde{M}$ is the set of asymptotic classes of rays. We refer the reader to \cite{ballmannnonpositivebook} and \cite{ballmannlectures} for a better understanding of this construction. 

The \emph{Busemann cocycle} at $\xi$ in $\partial\tilde{M}$ between $x$ and $y$ in $\tilde{M}$ is defined as
\begin{equation*}
	\beta_\xi (x,y)= \lim_{t\to+\infty} d(x,\sigma(t)) - d(y,\sigma(t)),
\end{equation*}
where $\sigma$ is any ray in the class $\xi$. If $v$ is a vector in $T^1\tilde{M}$, their \emph{points at infinity} $v_-$ and $v_+$ in $\partial \tilde{M}$ are respectively the asymptotic classes of the negative and positive rays tangent to $v$.
We can define the \emph{(unstable) horosphere} of $v$ as the set
\begin{equation*}
	H^u(v)=\{w\in T^1 \tilde{M} \,|\, w_-=v_-,\, \beta_{v_-}(\pi(v),\pi(w))=0 \},
\end{equation*}
where $\pi:T^1\tilde{M}\rightarrow \tilde{M}$ is the projection to the base. The point $v_-$ in $ \partial \tilde{M}$ is called the \emph{center} of the horosphere $H^u(v)$. Horospheres are $C^1 $ submanifolds of $T^1\tilde{M}$ of dimension $\dim M -1$.

We also use the notation $\partial^2\tilde{M}=(\partial\tilde{M}\times\partial \tilde{M})\setminus\Delta$, where $\Delta$ is the diagonal of $\partial\tilde{M}\times\partial \tilde{M}$, and define the map 
\begin{equation*}
\begin{matrix}
P: & T^1 \tilde{M} & \longrightarrow & \partial^2 \tilde{M} \times \RR  \\
& v & \longmapsto & (v_-,v_+,\beta_{v_-}(x_0,\pi(v)).
\end{matrix}
\end{equation*}
It is known that this map is a homeomorphism when the curvature is negatively pinched, but both the injectivity and the surjectivity may fail in the context of nonpositively curved rank $1$ manifolds. Nevertheless, restricted to rank $1$ vectors, $P$ is still a homeomorphism onto its image \cite{ballmannlectures}.

In the sequel, we identify $M$ with the quotient of $\tilde{M}$ by some discrete subgroup of isometries $\Gamma$, which is isomorphic to the fundamental group of $M$. The \emph{limit set} $\Lambda(\Gamma)$ is the set of accumulation points of an orbit $\Gamma x_0$, $x_0\in \tilde{M}$, in $\tilde{M}\cup \partial\tilde{M}$. It does not depend on the choice of $x_0$ and it is contained in the boundary at infinity $\partial\tilde{M}$. We also define the \emph{nonwandering set} $\Omega$ to be the set $\{ v\in T^1 \tilde{M} \,|\, v_-,v_+\in \Lambda(\Gamma) \}$. The name of this set comes from the fact that its projection to $T^1M$ is the topological nonwandering set of the geodesic flow. In order to have some complexity in the geodesic flow, we ask that $\Gamma $ (or $M$) is \emph{non-elementary}, which means that the limit set $\Lambda(\Gamma)$ is infinite.

A $\delta$\emph{-dimensional conformal density}, $\delta \ge 0$, is a family of finite Borel measures $\{\mu_x\}_{x\in \tilde{M}}$ on $\partial\tilde{M}$ supported by the limit set $\Lambda(\Gamma)$ and such that any two measures $\mu_x,\mu_y$, where $x,y\in\tilde{M}$, are equivalent and satisfy the relation $\frac{\dd \mu_y}{\dd \mu_x}(\xi)=\exp{(-\delta \beta_\xi(y,x))}$. In addition, we say that such a family of measures is $\Gamma$\emph{-invariant} if $\gamma_*\mu_x=\mu_{\gamma x}$ for all $\gamma\in\Gamma,\,x\in\tilde{M}$. The Patterson construction provides examples of invariant $\delta(\Gamma)$-dimensional conformal densities, where $\delta(\Gamma)$ is the critical exponent of $\Gamma$ (see \cite{Knieper97} for example).
Henceforth, we will fix a point $x_0\in\tilde{M}$ and denote by $\mu_{x_0}$ an element of an invariant $\delta $-dimensional conformal density. 

\subsection{Definition of the Bowen-Margulis measure}\label{Sec:notation}

We now define a product measure on the unitary tangent bundle of $\tilde{M}$, which will pass to the quotient $T^1M$, as Knieper did in \cite{Knieper98} for a compact manifold. Consider the set of geodesic endpoints $E(\tilde{M}):=\{(v_-,v_+)\in \partial^2\tilde{M}\,|\,v\in T^1\tilde{M} \}$. For every $(\xi,\eta)\in E (\tilde{M})$, the set $\pi(P^{-1}(\{(\xi,\eta)\}\times \RR))$ is nonempty, and it has been shown to be a flat totally geodesic submanifold of $\tilde{M}$ \cite{Eberlein96}. In fact, it is either a single geodesic or a flat totally geodesic submanifold of dimension at least $2$. No matter what form the submanifold $\pi(P^{-1}(\{(\xi,\eta)\}\times \RR))$ takes, we denote by $\vol$ its induced volume measure.

Firstly, we define a measure $\bar{\mu}$ on $E(\tilde{M})$, which we extend to $\partial^2\tilde{M}$, by its density
\begin{equation*}
\dd \bar{\mu}(\xi,\eta)= e^{\delta (\beta_\xi(x_0,p_{\xi,\eta})+ \beta_\eta(x_0,p_{\xi,\eta}))}\dd \mu_{x_0}(\xi)\,\dd \mu_{x_0}(\eta),
\end{equation*}
where $p_{\xi,\eta}$ is any point in $\pi(P^{-1}(\{(\xi,\eta)\}\times \RR))$. The definition does not depend on the choices of $p_{\xi,\eta}$, and $\bar{\mu}$ is invariant under the diagonal action of $\Gamma$ on $\partial^2\tilde{M}$. 
Now, the measure $\mu$ on $T^1 \tilde{M}$ associated to $\mu_{x_0}$ gives the value

\begin{equation}\label{definition_measure}
	\mu(A)=\int_{\partial^2\tilde{M}} \vol(\pi(P^{-1}   (\{(\xi,\eta)\}\times \RR)\cap A))\,\dd \bar{\mu}(\xi,\eta)
\end{equation}
to a Borel subset $A\subset T^1\tilde{M}$. It is clear that this measure is both $\Gamma$ and $g_t$-invariant. The $g_t$-invariant measure obtained on the quotient $T^1 M$ will also be denoted by $\mu$, for simplicity of notation.

In the sequel, we assume that $M$ contains a closed rank $1$ geodesic. Many authors have studied the ergodic properties of this measure in this setting. G. Link and J. C. Picaud gave a version of the Hopf-Tsuji-Sullivan dichotomy (see \cite{LinkPicaud} and \cite{Link1}). In this article we will always be in the conservative case of the dichotomy: $\Gamma$ is of divergence type, the radial limit set has full $\mu_{x_0}$-measure and the system $(T^1M,g_t,\mu)$ is conservative and ergodic. Furthermore, in this case there is a unique conformal density $\mu_{x_0}$ and its dimension is the critical exponent $\delta=\delta (\Gamma)$ of $\Gamma$, this measure $\mu_{x_0}$ has no point masses, and the measure class of $\mu_{x_0}$ is ergodic under the action of $\Gamma$. We refer to $\mu$ as the Bowen-Margulis measure. Whenever $M$ is compact, it turns out that the Bowen-Margulis measure $\mu$ is the unique measure of maximal entropy up to a multiplicative constant as proved by G. Knieper in \cite{Knieper98}.

Let $\mathcal{H}$ be the set of unstable horospheres in $T^1 \tilde{M}$ and let $\Gamma $ act on $\mathcal{H}$. There is a simple identification of $\mathcal{H}$ by its point at infinity and the value of the Busemann cocycle: the map
\begin{equation*}
	\begin{matrix}
	\mathcal{H} & \longrightarrow & \partial \tilde{M} \times \RR  \\
	H^u(v) & \longmapsto & (v_-,\beta_{v_-}(x_0,\pi(v)).
	\end{matrix}
\end{equation*}
is bijective. This allows us to define a $\Gamma$-invariant measure $\hat{\mu}$ on the space of horospheres $ \mathcal{H}$ by the density
\begin{equation*}
	\dd \hat{\mu}(H^u(\xi,t))=e^{-\delta t}\dd \mu_{x_0}(\xi)\dd t,
\end{equation*}
where $H^u(\xi,t)$ is the unstable horosphere with coordinates $(\xi,t)\in \partial\tilde{M}\times \RR$.

We can now define a family $\{\mu_{H}\}_{H\in \mathcal{H}}$ of measures on each horosphere that has good properties, analogous to the negative curvature case. For each horosphere $H$, we consider the projection to the positive endpoint $P_H:H\rightarrow \partial \tilde{M}\setminus\{\xi\}$, where $\xi$ is the center of $H$. Let us treat a point as a $0$-manifold, for the sake of simplicity. For any vector $v$ in $T^1\tilde{M}$, the set of base points of vectors on the horosphere $H^u(v)$ that are also positively asymptotic to $v$, i.e. the set $\pi (P_{H^u(v)}^{-1}(v_+))$, is a totally geodesic submanifold of $\pi(P^{-1}((v_-,v_+)\times \RR))$. We denote its volume measure by $\vol$ with the convention that it is the delta measure when the submanifold consists of a single point. For each $\eta \in \partial \tilde{M}\setminus\{v_-\}$, we choose $w\in P_{H^u(v)}^{-1}   (\eta))$ and write $\phi_v(\eta)=e^{\delta \beta_{\eta}(x_0,\pi(w))}$, which in fact only depends on $\eta$, but not on $w$. The measure $\mu _{H^u(v)}$ assigns the value
\begin{equation*}
	\mu_{H^u(v)}(A)=\int_{\partial\tilde{M}\setminus\{v_-\}} \vol(\pi (P_{H^u(v)}^{-1}(\eta)\cap A))
	\phi_v(\eta)
	\,\dd \mu_{x_0}(\eta)
\end{equation*}
to a subset $A\subset H^u(v)$.

We enumerate the main properties of these measures that follow from the definition.
\begin{enumerate}[label={(\roman*)}]
	\item If $w\in H^u(v)$, then $\mu_{H^u(v)}=\mu_{H^u(w)}$. Hence, we can speak of a family $\{\mu_{H}\}_{H\in \mathcal{H}}$ of measures on each horosphere.
	\item They are $\Gamma$-invariant; i.e. for all $\gamma$ in $\Gamma$ and all $H$ in $\mathcal{H}$, we have $\gamma_*\mu_H=\mu_{\gamma H}$.
	\item They are exponentially expanded by the geodesic flow: $\mu_{g_tH}=e^{\delta t}(g_t)_*\mu_{H}$.
	\item The measure $\mu$ is the product of $\{\mu_{H}\}_{H\in \mathcal{H}}$ by $\hat{\mu}$: for all $A\subset T^1 \tilde{M}$,
	\begin{equation}\label{product_structure}
		\mu(A)=\int_{\mathcal{H}} \mu_H(A\cap H)\, \dd \hat{\mu}(H).
	\end{equation}
\end{enumerate}

For our purpose, we assume that the Bowen-Margulis measure $\mu$ on the space $T^1M$  is finite, hence the geodesic flow is conservative, according to the Poincaré recurrence theorem, and there is only one conformal density $\mu_{x_0}$. Our goal is to find an equidistribution result in the sense that the $\mu_H$-averages of functions on a horosphere $H$ tend to the $\mu$-averages on the whole space. The starting point is always the mixing property of the geodesic flow with respect to the Bowen-Margulis measure $\mu$. The next result says that this property is equivalent to the topological mixing of the geodesic flow on $\Omega$. We do not know if this equivalence has been stated in this generality, although it can be expected and the main part of the work is already published.

There is a third equivalent property related to the length of the closed geodesics, analogous to what happens in negative curvature. We define the \emph{rank $1$ length spectrum} as the set of lengths of rank $1$ closed geodesics. We say that the rank $1$ length spectrum is \emph{non-arithmetic} if the rank $1$ length spectrum generates a dense subgroup of $\RR$.

\begin{theorem}
	Let $M$ be a rank $1$ nonpositively curved non-elementary complete connected Riemannian manifold. Assume that the Bowen-Margulis measure $\mu$ is finite. Then the following are equivalent:
	
	\begin{enumerate}[label={(\roman*)}]
		\item The geodesic flow $g_t$ is topologically mixing on the nonwandering set $\Omega$.
		\item The geodesic flow $g_t$ is mixing with respect to the Bowen-Margulis measure $\mu$.
		\item The rank $1$ length spectrum is non-arithmetic.
	\end{enumerate}
\end{theorem}

\begin{proof}
	(ii)$\implies$ (i) The mixing property with respect to a measure implies the topological mixing on the support of the measure. In our case, the support of $\mu$ is the nonwandering set $\Omega$, so the implication is proved.
	
	(i)$\implies$ (iii) We reproduce the reasoning used in negative curvature \cite{Dalbo00}. Since the set of rank $1$ vectors is open \cite{ballmannlectures},
	 we can find a closed ball $B$ of certain radius only containing rank $1$ vectors. Let $\varepsilon >0$ be a given number. We apply the closing lemma for the rank $1$ set \cite[Proposition 4.5.15]{Eberlein96}: there exists constants $T_0>0$ and $\delta>0$ such that, for every $v\in B$ and $t\ge T_0$ with $d_1(v,g_t(v))\le \delta$, there exists a periodic rank $1$ vector $v'$ at distance $d_1(v,v')\le \varepsilon$, where the period $t'$ of $v$ satisfies $\modul{t-t'}<\varepsilon$.
	
	There exists a nonempty open subset $U$ of $\Omega$ of diameter smaller than $\delta$ and such that $U\subset B$. Since the geodesic flow on $\Omega$ is topologically mixing, there exists a number $T\ge T_0$ such that for all $t\ge T$, we have $U\cap g_t(U)\not= \emptyset$. In particular, there is a rank $1$ vector $v$ in $B$ satisfying $d_1(v,g_t(v))\le \delta $. Hence, for each $t \ge T$, there exists a periodic rank $1$ vector of period in $[t-\varepsilon,t+\varepsilon]$. Since $\varepsilon$ is arbitrary, this proves that the rank $1$ length spectrum is non-arithmetic.
	
	(iii)$\implies$ (ii) This implication may be the hardest, but it is essentially done in the proof of Theorem 2 in \cite{Babillot1}, asserting that the geodesic flow is mixing with respect to $\mu$ on a compact manifold. All the arguments work for a rank $1$ manifold with finite Bowen-Margulis measure, but at the end, instead of applying the compactness, we can use the assumption of non-arithmeticity of the length spectrum.
\end{proof}

To summarize, in all statements $M$ is a non-elementary nonpositively curved complete connected manifold with a closed geodesic of rank $1$ such that the geodesic flow is topologically mixing on $\Omega$ and such that the Bowen-Margulis measure $\mu$ is finite. 

\subsection{Equidistribution of horocycles}

We start with a local result showing that near rank $1$ vectors there is equidistribution: for a function $f:T^1M\rightarrow \RR$, the average on a horosphere of its lift $\tilde{f}:T^1\tilde{M}\rightarrow \RR$ pushed by the geodesic flow converges to the average of $f$ with respect to the Bowen-Margulis measure.

\begin{proposition}\label{prop_local_ed}
	Let $M$ be a nonpositively curved non-elementary complete connected Riemannian manifold with a closed rank $1$ geodesic. Assume that the geodesic flow $g_t$ on $T^1M$ is topologically mixing on $\Omega$ and that the Bowen-Margulis measure $\mu$ is finite. Then, for every rank $1$ vector $v\in \Omega \subset T^1 \tilde{M}$, there exists an open subset $U$ of $H^u(v)$ containing $v$ which is equidistributed under the action of the geodesic flow; i.e. for every bounded and uniformly continuous function $f$ on $T^1 M$ and every Borel neighborhood $V\subset U$ of $v$ we have
	\begin{equation*}
		\frac{1}{\mu_{H^u(v)}(V)}\int_V \tilde{f}\circ g_t \, \dd \mu_{H^u(v)} \xrightarrow[t \to +\infty]{} \frac{1}{\mu(T^1 M)} \int_{T^1 M}f \,\dd \mu .
	\end{equation*}
	
\end{proposition}
	
\begin{proof}
	We follow the same strategy as M. Babillot in \cite{Babillot1}, which consists in approximating the integral on a piece of horosphere by the integral of the same function on a box around that piece, and then use the mixing property of the geodesic flow with respect to $\mu$. The added difficult is to find a box with a good system of coordinates, which is done by avoiding the vectors of higher rank.
	
	Let $v$ be a rank $1$ vector in $\Omega$ and denote its horosphere by $H$. 
	From \cite[Lemma III.3.1]{ballmannlectures}
	 we know that there exist disjoint connected neighborhoods $A_1$ and $A_2$ of $v_-$ and $v_+$, respectively, in $\partial \tilde{M}$ such that: for every $(\xi,\eta)\in A_1\times A_2$ there exists a unique geodesic from $\xi$ to $\eta$, and it has rank $1$. This allows us to consider a coordinate neighborhood of $v$ via the map $P$ of the form $A_1\times A_2 \times \RR$.
	
	We claim that the proposition is true with $U=P_H^{-1}( A_2)$. Consider any neighborhood $V\subset U$ of $v$ and write $V_+:=\{w_+\,|\,w\in V\}$ for its projection to the boundary at infinity $\partial \tilde{M}$. Since $v$ is nonwandering, its endpoints are in the limit set, and this guarantees that $V_+$ and $V$ have positive measure. We notice that the integral on $V$ of a function $h$ of $T^1 \tilde{M}$ can be written in coordinates as
\begin{equation*}
	\int_V h\,\dd\mu_{H^u(v)}=\int_{V_+}h(v_-,\eta , t_0)e^{\delta \beta_\eta (x_0,\pi (v_-,\eta, t_0))} \dd\mu_{x_0}(\eta),
\end{equation*}
where $t_0$ has the value $\beta_{v_-}(x_0,\pi(v))$, because the volume $\vol$ is always $1$ on rank $1$ vectors. This is because $P^{-1}   (\{(v_-,\eta,t_0)\}$ consists of just one vector when $v$ is of rank $1$. Otherwise, the flat strip theorem \cite[Corollary I.5.8(ii)]{ballmannlectures} asserts that $v$ bounds a flat totally geodesic surface, which is not possible for a rank $1$ vector.

Given $\varepsilon>0$, we can find a small connected neighborhood $B\subset A_1$ of $v_-$ and a number $r>0$ such that:
\begin{enumerate}[label={(\roman*)}]
	\item \label{expapprox}$\forall \xi \in B,\, \forall \eta \in V_+, \quad 1-\varepsilon\le 
	e^{\delta \beta_\eta (\pi (v_-,\eta, t_0),\,\pi (\xi,\eta, t_0))}
	\le 1+\varepsilon$,
	\item \label{fapprox}$\forall (\xi,\eta) \in B\times V_+, \forall s\in [-r,r], \forall t\ge 0,
	\quad
	| \tilde{f}(v_-,\eta , t_0+t)-\tilde{f}(\xi,\eta , t_0+t+s)|<\varepsilon$.	
\end{enumerate}
The first property follows from the continuity of the map $P$ on the coordinate neighborhood, and the continuity of the projection $\pi$ and of the Busemann function. We use that $V_+$ is relatively compact to assert that the inequality holds uniformly in $\eta\in V_+$. For property \ref{fapprox}, we apply the uniform continuity of $\tilde{f}$, and then we choose $B$ and $r$ so that the points $(v_-,\eta , t_0)$ and $(\xi,\eta , t_0+s)$ are close enough for $\xi \in B$ and $s\in [-r,r]$, uniformly in $\eta \in V_+$. Since these points are in the same weak stable leaf, the distance between them does not increase when they are pushed by the geodesic flow, which allows us to deduce the above property for all $t\ge 0$. Again the condition $v\in \Omega $ implies that $v_-\in\Lambda(\Gamma)$, which ensures that $B$ has positive measure.

These estimates allow one to compare the average of $\tilde{f}\circ g_t$ on the set $V$ with respect to $\mu_{H^u(v)}$ and the average of the same function on the box of the form $P^{-1}(B\times V_+\times [t_0,t_0+r])$ with respect to the measure $\mu$ by means of the product structure of $\mu$ (Equation \ref{definition_measure}). More precisely, for all nonnegative $t$,

\begin{equation*}
\left[\frac{ \int_V \tilde{f}\circ g_t \, \dd \mu_{H^u(v)} }{ \mu_{H^u(v)}(V) }-\varepsilon 
\right]
\frac{1-\varepsilon}{1+\varepsilon}
\le
\frac{ \int_{P ^{-1}(B\times V_+\times [t_0,t_0+r])} \tilde{f}\circ g_t \, \dd \mu }{\mu(P^{-1}(B\times V_+\times [t_0,t_0+r]))} 
\le
\end{equation*}
\begin{equation*}
	\left[\frac{ \int_V \tilde{f}\circ g_t \, \dd \mu_{H^u(v)} }{ \mu_{H^u(v)}(V) }+\varepsilon 
	\right]
	\frac{1+\varepsilon}{1-\varepsilon}.
\end{equation*}

Moreover we may assume that the neighborhood $P^{-1}(B\times A_2\times [t_0,t_0+r])\subset T^1\tilde{M}$ is homeomorphic to its projection on the unit tangent bundle of the manifold $M$. Then, since the geodesic flow is mixing with respect to $\mu$, the average of $\tilde{f}\circ g_t$ in $P^{-1}(B\times V_+\times [t_0,t_0+r])$ converges to $\frac{1}{\mu{(T^1 M)}}\int f \dd \mu $ when $t$ goes to infinity. We have thus shown the equidistribution of $U$.

\end{proof}	

To deduce a global result, we need to understand what happens on vectors of rank different from $1$, and the next two lemmas will be crucial. The \emph{unstable manifold} of $v$ in $T^1\tilde{M}$ is the set
\begin{equation*}
	W^u(v)=\{w\in T^1\tilde{M}\,|\, d_1(g_t(v),g_t(w))\rightarrow 0,\,t\to -\infty\}.
\end{equation*}
$W^u(v)$ is a subset of the unstable horosphere $H^u (v)$, but they are not necessarily equal in nonpositive curvature.

\begin{lemma}\label{lemma1}
	Let $M$ be a rank $1$ nonpositively curved non-elementary complete connected Riemannian manifold. If $v$ is a rank $1$ recurrent vector in $T^1\tilde{M}$, then its unstable horosphere coincides with its unstable manifold, $H^u(v)=W^u(v)$, and it consists of rank $1$ vectors exclusively.
\end{lemma} 
\begin{proof}
	The fact that the unstable manifold and the horosphere coincide is already proved in \cite[Proposition 4.1]{Knieper98}. Let $w$ in $W^u(v)$ and $r$ its rank, we will see that $r$ is $1$. Since $v$ is negatively recurrent there exist a sequence $t_n\rightarrow -\infty$ and isometries $\gamma_n\in \Gamma$ such that $\gamma_n(g_{t_n}(v))\rightarrow v$ when $n\to\infty$. Now we have
	\begin{equation*}
		d_1(v,\gamma_ng_{t_n}(w))\le  d_1(v,\gamma_ng_{t_n}(v))+d_1(g_{t_n}(v),g_{t_n}(w))\longrightarrow 0
	\end{equation*}
	and the rank of $\gamma_n g_{t_n}(w)$ is the same as the rank of $w$, $r$. Since $v$ is a limit of vectors of rank $r$ and the rank function is upper semi-continuous, we deduce $r\le \rank v = 1$. 
\end{proof}

\begin{lemma}\label{lemma2}
	Let $M$ be a nonpositively curved non-elementary complete connected Riemannian manifold with a closed rank $1$ geodesic. Assume that the Bowen-Margulis measure $\mu$ is finite and that the geodesic flow $g_t$ on $T^1M$ is ergodic with respect to the $\mu$. Then, for every horocycle $H$, the set of vectors in $H$ of rank equal or bigger than $2$ is $\mu_H$-negligible.
\end{lemma}

\begin{proof}
	Let $Rec^1\subset T^1\tilde{M}$ be the set of rank $1$ vectors which are recurrent under $g_t$ on the quotient $T^1 M$ and $S\subset T^1\tilde{M}$ be the set of vectors of rank $2$ or higher. We claim that the projections to the boundary of these two sets are disjoint, $Rec^1_+\cap S_+=\emptyset$. Otherwise, there are vectors $v\in Rec^1$ and $w\in S$ such that $v_+=w_+$. By Lemma $\ref{lemma1}$, the unstable horosphere of $-v$ only contains vectors of rank $1$. The geodesic associated to $-w$ intersects this horosphere $H^u(-v)$ (Figure \ref{fig:pic2}), so $w$ should have rank $1$, which is a contradiction. 
	
	\begin{figure}[h]
		\centering
		\includegraphics{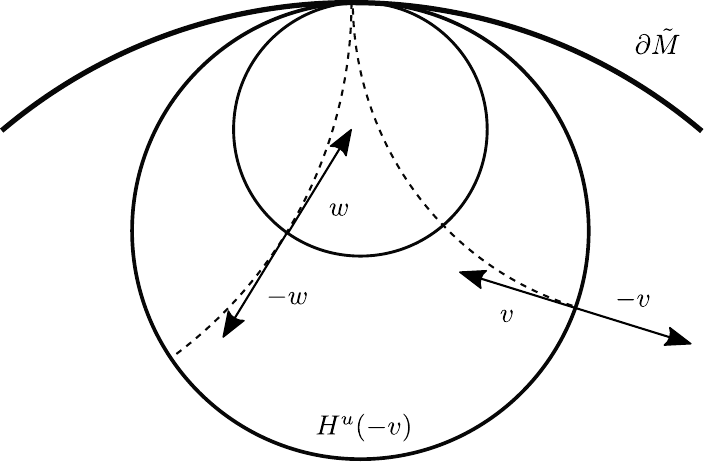}
		\caption{ \label{fig:pic2}
			Vectors $v$ and $w$ of the proof.}
	\end{figure}
	
	Around a nonwandering rank $1$ vector there is a neighborhood only consisting of rank $1$ vectors, and this neighborhood has positive measure because it intersects the support of $\mu$. By hypothesis, the manifold $M$ contains a closed rank $1$ geodesic, which is an example of nonwandering rank $1$ geodesic. The set of rank $1$ vectors  has positive measure, and it is invariant under the geodesic flow. So the set of rank $1$ vectors has full measure because of the ergodicity of $\mu$. In consequence, the set of rank $1$ recurrent vectors $Rec^1$ has also full $\mu$-measure in view of Poincaré recurrence theorem. By the product structure of $\mu$, we see that $Rec^1_+$ has positive $\mu_{x_0}$-measure. Finally, $Rec^1_+$ is a $\Gamma$-invariant set, so we deduce that $Rec^1_+$ has full $\mu_{x_0}$-measure because $\Gamma$ acts ergodically.
	
	Therefore, $S_+$ is negligible. The endpoints of higher rank vectors in $H^u(v)$ are clearly in $S_+$ and, using the definition of the measure on the horosphere, we obtain $\mu_{H^u(v)}(S\cap H^u(v))=0$.
\end{proof}

We can finally prove Theorem \ref{theorem_timeed}, which we have reformulated in terms of horospheres on the universal cover $\tilde{M}$. On the horospheres centered at the limit set, every open set with positive and finite measure is equidistributed (Figure \ref{fig:pic3}). Being positive is equivalent to having a nonwandering rank $1$ vector. In particular, all relatively compact neighborhoods of nonwandering rank $1$ vectors are equidistributed. 

\begin{figure}[h]
	\centering
	\includegraphics{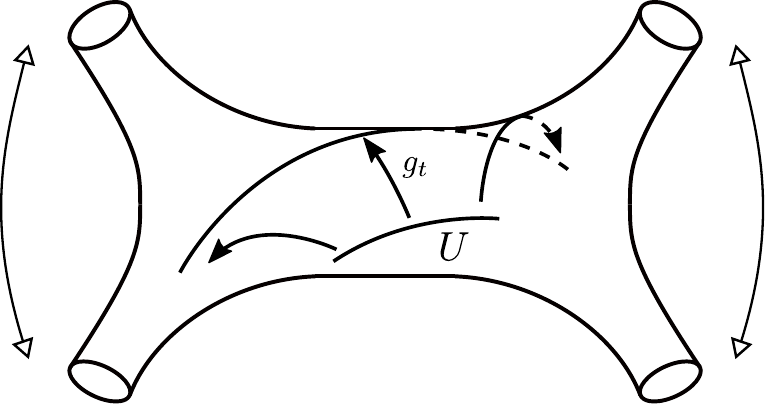}
	\caption{ \label{fig:pic3}
		The average of $f$ on the image of an open subset $U$ of a horosphere $H$ by the geodesic flow $g_t$ with respect to $\mu_H$ tends to the average of $f$ with respect to $\mu$.}
\end{figure}

\begin{theorem}
	Let $M$ be a nonpositively curved non-elementary complete connected Riemannian manifold  with a closed rank $1$ geodesic. Assume that the geodesic flow $g_t$ on $T^1M$ is topologically mixing on $ \Omega$ and that the Bowen-Margulis measure $\mu$ is finite. Then, for every horosphere $H\subset T^1 \tilde{M}$ centered at $\Lambda(\Gamma)$, every open subset $U$ of $H$ of finite but positive $\mu_H$-measure is equidistributed under the action of the geodesic flow; i.e. for every bounded and uniformly continuous function $f$ on $T^1M$, we have
	\begin{equation*}
	\frac{1}{\mu_{H}(U)}\int_U \tilde{f}\circ g_t \, \dd \mu_{H} \xrightarrow[t \to +\infty]{} \frac{1}{\mu(T^1 M)}\int_{T^1 M}f \,\dd \mu .
	\end{equation*}
\end{theorem}

\begin{proof}
	We first observe that the set $U^1$ of rank $1$ vectors in $U$ is open in $H$, because the set of rank $1$ vectors is open in $T^1 \tilde{M}$ \cite{ballmannlectures}. By Lemma \ref{lemma2}, the set $U^1$ has full measure in $U$, so the averages on the two sets are the same. Next, we use the fact that $\mu_H$ is a Radon measure: given a number $\varepsilon>0$, there exists a compact subset $K\subset U^1$ such that $\mu_H(U^1\setminus K)<\varepsilon $. 
	
	Since $\Omega $ is closed, $L=K\cap \Omega $ is again compact, and $L$ has full measure in $K$, because vectors outside of $\Omega$ are not in the support. We want to show that $L$ is equidistributed. Proposition \ref{prop_local_ed} gives an equidistributed open neighborhood $U_v$ of each vector $v$ in $L$. The set $L$ can be covered by finitely many $U_v$ because it is compact. We can cut these sets where they intersect to obtain  a family $\{V_i\}_{1\le i\le n}$ of equidistributed pairwise disjoint Borel sets whose union contains $L$, thanks to the fact that the subsets of $U_v$ are equidistributed too.
	
	If we let $\lambda :=\int f \dd \mu /\mu(T^1 M)$, the set $V:=V_1\cup\dots \cup V_n$ is equidistributed because
	\begin{equation*}
	\frac{\int_V \tilde{f}\circ g_t \, \dd \mu_{H}}{\mu_{H}(V)} 
	=\frac{\sum_{i=1}^n \int_{V_i} \tilde{f}\circ g_t \, \dd \mu_{H}}{\mu_{H}(V)} 
	\xrightarrow[t \to +\infty]{} 
	\frac{\sum_{i=1}^n \mu_H(V_i) \lambda}{\mu_H(V)}
	=\lambda.
	\end{equation*}
	On the other hand, we have $\mu_{H}(U\setminus V)<\varepsilon$, so
	\begin{equation*} 
	\left| \frac{1}{\mu_{H}(U)}\int_{U} \tilde{f}\circ g_t \, \dd \mu_{H} -
	\frac{1}{\mu_{H}(V)}\int_V \tilde{f}\circ g_t \, \dd \mu_{H} \right|
	\le \frac{2\varepsilon \norm{f}_\infty}{\mu_H(U)}
	\end{equation*}
	for all $t\ge0$. This proves that $U$ is equidistributed as well.
\end{proof}

\section{Unique ergodicity of the horocyclic flow}

\subsection{Surfaces with nonpositive curvature}

In the second part of the article we restrict our attention to surfaces. Our goal is to define a flow that preserves the Bowen-Margulis measure and whose orbits are horocycles. Then we would like to interpret the equidistribution of horocycles in terms of the ergodic properties of this flow. The idea is to define the parametrization of the flow by the measures on the horocycles as in the negative curvature case \cite{Marcus1}. However, the presence of flat pieces of horocycle makes impossible to define globally a continuous flow with this method. We have found a subset ${\Sigma}$ of the unitary tangent bundle which excludes the horocycles causing trouble, like that of Figure \ref{fig:pic1}, and which is topologically and metrically large. We will define a parametrization of the horocyclic flow on ${\Sigma}$ and prove that it is uniquely ergodic.

In this section, $M$ is a nonpositively curved non-elementary orientable surface with a closed rank $1$ geodesic and the Bowen-Margulis measure $\mu$, constructed as before, is assumed to be finite. We will further assume that $M$ satisfies the duality condition, which means that every vector of $T^1M$ is nonwandering, or equivalently we assume that $\Lambda(\Gamma)=\partial\tilde{M}$. Under these hypothesis, the geodesic flow is topologically mixing \cite[Theorem 6.3]{Eberlein73}, so it is also mixing with respect to the Bowen-Margulis measure. The duality condition is satisfied if $M$ has finite Riemannian volume, as an application of the Poincaré recurrence theorem. Just to mention, a nonpositively curved non-elementary rank $1$ manifold satisfying the duality condition contains automatically a closed rank $1$ geodesic.

Moreover, we know that any two distinct points in the boundary at infinity can be connected by a geodesic. It follows from the fact that, for a nonflat surface $M$ with the duality condition, the universal cover $\tilde{M}$ satisfies the visibility axiom \cite[Proposition 2.5]{Eberlein79}. Therefore, the map $P$ is surjective.

We notice that an orientation of the boundary at infinity $\partial \tilde{M}$ induces an orientation to each horocycle in $\mathcal{H}$. One vector $v\in T^1 \tilde{M}$ divides its horocycle $H^u(v)$ in two connected sets, one in the positively oriented direction, $H^u_+(v)$, and the other in the negatively oriented, $H^u_-(v)$. The group of isometries $\Gamma$ is orientation-preserving because $M$ is orientable. In consequence, horocycles on $T^1\tilde{M}$ descend to $T^1M$ as oriented immersed curves.

Horocycles are diffeomorphic to the real line. Let $H$ be a horocycle of $T^1 \tilde{M}$. The interval $(v,w)\subset H$ between two vectors $v,w\in H$ is the connected subset bounded by $v$ and $w$. The map $P_H:H\rightarrow \partial\tilde{M}\setminus\{\xi\}$, where $\xi$ is the center of $H$, which projects a vector to its positive endpoint, is continuous and surjective. We also observe that $P_H(v)=P_H(w)$, with $v\not =w$, implies, according to the flat strip theorem, that the curvature vanishes on the strip $\pi(\cup_{t\in\RR} g_t((v,w)))$. Such an interval $(v,w)$ will be called \emph{flat piece of horocycle} (see Figure \ref{fig:pic4}). It is clear that $H$ does not contain any flat piece if and only if $P_H$ is injective, in which case $P_H$ is also a homeomorphism.

\begin{figure}[h]
	\centering
	\includegraphics{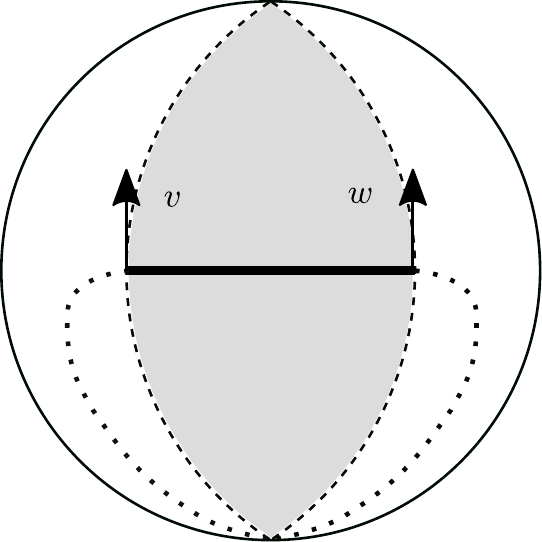}
	\caption{ \label{fig:pic4}
		Universal cover of the surface ${M}$ with a region where the curvature vanishes (shadowed region). We represent an unstable horocycle with a flat piece.}
\end{figure}

\subsection{Definition of the horocyclic flow on a certain subset of $T^1 M$}
Next, we define a subset of the unitary tangent bundle $T^1 \tilde{M}$ of $\tilde{M}$ and we study the properties of its horocycles and their associated measures. Let $\tilde{\Sigma} \subset T^1\tilde{M}$ denote the set of vectors whose horocycle contains a rank $1$ recurrent vector, that is to say,
\begin{equation*}
	\tilde{\Sigma}= \bigcup_{v\in Rec^1 }H^u(v).
\end{equation*}
This set is invariant under $\Gamma$, under the geodesic flow and under the horocyclic foliation, in the sense that $\tilde{\Sigma}$ contains a horocycle $H$ as soon as it contains one vector of $H$. Our set $\tilde{\Sigma}$ contains a $G_\delta $-dense set, namely the set of rank $1$ recurrent vectors $Rec^1$. The latter is the intersection of the set of rank $1$ vectors, which is open and dense \cite[Corollary III.3.8]{ballmannlectures}, with the set of recurrent vectors, which is $G_\delta$-dense when all the vectors of $T^1M$ are nonwandering. The set $\tilde{\Sigma}$ also has full $\mu$-measure. By Lemma \ref{lemma1}, all the vectors in $\tilde{\Sigma}$ have rank $1$ and each horocycle $H\subset\tilde{\Sigma}$ coincides with the unstable manifold. This also implies that the horocycles in $\tilde{\Sigma} $ do not contain any flat pieces of horocycle.

In the next lemma, which will be needed later, we prove a sort of continuity of the measures on the horocycles contained in $\tilde{\Sigma}$.
\begin{lemma}\label{continuity_measure}
	The map
	\begin{equation*}
	\begin{matrix}
	\{(v,w)\in \tilde{\Sigma}\times \tilde{\Sigma}\,|\,w\in H^u(v)\} & \longrightarrow & \RR \\
	(v,w) & \longmapsto & \mu_{H^u(v)}((v,w))
	\end{matrix}
	\end{equation*}
	is continuous.
\end{lemma}

\begin{proof}
	Let $v$ and $w$ two points in $\tilde{\Sigma}$ sharing a horocycle. The function on a pair of points $(v',w')$ close to $(v,w)$ can be written as the integral
	\begin{equation*}
		\mu_{H^u(v')}((v',w'))=\int_{(v'_+,w'_+)}e^{\delta \beta_\eta (x_0,\pi(P^{-1}_{H^u(v')}(\eta)))} \dd \mu_{x_0}(\eta).
	\end{equation*}
	Given $\varepsilon>0$, we can suppose that $\mu_{x_0}((v'_+,w'_+)\triangle (v_+,w_+))<\varepsilon$.	We estimate the Busemann cocycles on $H^u(v)$ and $H^u(v')$ as we did in the proof of Proposition \ref{prop_local_ed}: for all $\eta \in (v_+,w_+)$ we have
	\begin{equation*}
		1-\varepsilon\le 
		e^{\delta \beta_\eta (\pi (P^{-1}_{H^u(v)}(\eta)),\,\pi (P^{-1}_{H^u(v')}(\eta)))}
		\le 1+\varepsilon,
	\end{equation*}
	provided that $v'$ is close to $v$.
	We also observe that $\exp({\delta \beta_\eta (x_0,\,\pi (P^{-1}_{H^u(v')}(\eta)))})$ is bounded by a constant $K$ for $v'$ in a neighborhood of $v$ and $\eta$ in a neighborhood of $v_+$ or $w_+$. With these approximations we get the inequalities
	\begin{equation*}
		(1-\varepsilon)\mu_{H^u(v)}((v,w)) - K  \varepsilon \le 
		\mu_{H^u(v')}((v',w'))
		\le (1+\varepsilon)\mu_{H^u(v)}((v,w)) + K \varepsilon,
	\end{equation*}
	which show the continuity at $(v,w)$.
\end{proof}

Let us remark that only the fact that all the vectors in the domain $\tilde{\Sigma}$ have rank $1$ was used in the previous proof. Next, we state and prove some properties of individual measures on horocycles that will help later to define the parametrization.

\begin{lemma}\label{properties_horocycles}
	Let $H$ be a horocycle of $T^1\tilde{M}$ and $v\in H$:
	\begin{enumerate}[label={(\roman*)}]
		\item The measure $\mu_H$ has no point masses.
		\item If $H$ does not contain any flat piece, then $\mu_H$ is of full support in $H$.
		\item The measure $\mu_H$ is finite on compact sets.
		\item If $v$ is in $\tilde{\Sigma}$, then the half horospheres $H^u_+(v)$ and $H^u_-(v)$ have infinite measure.
	\end{enumerate}
	
\end{lemma}

\begin{proof}
	(i) We know that $\mu_{x_0}$ has no point masses. If $w\in H$, $\mu_{x_0}(\{v_+\})=0$ directly implies that $\mu_H(\{v\})=0$.
	
	(ii) If $U\subset H$ is an open nonempty subset, $P_H(U)$ is also open and nonempty. So $\mu_{x_0}(P_H(U))>0$ because its support is $\Lambda(\Gamma)=\partial\tilde{M}$. Then $\mu_H(U)=\int_{P_H(U)}\phi_v \dd \mu_{x_0}>0$.
	
	(iii) If $K\subset H$ is compact, $P_H(K)$ is also compact. The function $\phi_v$ is bounded on $P_H(K)$. The volume part of the integral is bounded by the length of $K$. Then it is clear that $\mu_H(K)$ is finite.
	
	(iv) By (iii) it is clear that, for every $w\in H^u(v)$, the measure of $H^u_+(v)$ is infinite if and only if so is the measure of $H^u_+(w)$. So we can assume that $v$ is in $Rec^1$.
	
	Let $B^u(w,r)$ denote the open ball in $H^u(w)$ of center $w$ and radius $r>0$. The balls $B^u(w,1)$ have two boundary points $a_w,b_w\in H^u(w)$ that depend continuously on $w$ so that $B^u(w,1)=(a_w,b_w)$. In view of Lemma \ref{continuity_measure}, the function $w\mapsto \mu_{H^u(w)}((a_w,w))$ is continuous. The continuity at $v$ implies that there exists a neighborhood $U$ of $v$ in $\tilde{\Sigma}$ such that for all $w\in U$
	\begin{equation}\label{measure_bound}
		\mu_{H^u(w)}((a_w,w))\ge\frac{1}{2}\mu_{H^u(v)}((a_v,v)).
	\end{equation} 
	The inequality is in fact valid on $\cup_{\gamma \in \Gamma} \gamma U$ because the family of measures is $\Gamma$-invariant. 

	Since $v$ is recurrent, there is a sequence $t_k$ converging to $-\infty$ and isometries $\gamma_k \in \Gamma$ such that the the distance $d_1(g_{t_k}v, \gamma_k v)$ goes to $0$. For $k$ big enough, the vector $g_{t_k}v$ is in $\gamma_k U$, so Equation \ref{measure_bound} remains true if we replace $w$ by $g_{t_k}v$. Let $a_k, b_k$ be the points in $H^u(g_{t_k}v)$ such that $B^u(g_{t_k}v,1)=(a_k,b_k)$. Using the fact that the measures on horocycles expand exponentially, we obtain
	\begin{equation*}
		\mu_{H^u(v)}((g_{-t_k}a_k,v))=e^{-t_k}\mu_{H^u(g_{t_k}v)}((a_k,g_{t_k}v))\ge 
		\frac{1}{2} e^{-t_k} \mu_{H^u(v)}((a_v,v)).
	\end{equation*}
	This shows that in one half-horocycle there are subsets of arbitrarily large measure. We proceed analogously for the other half-horocycle, with $b_k$ instead of $a_k$.	
\end{proof}

We can now define a suitable parametrization of the horocyclic flow on the set $\tilde{\Sigma}$. Given $v\in \tilde{\Sigma}$, we consider the function $m_v:H^u_+(v)\rightarrow (0,+\infty)$ defined by $m_v(w):=\mu_{H^u(v)}((v,w))$. The map $m_v$ is well defined by properties (ii) and (iii) of Lemma \ref{properties_horocycles}, is continuous by (i), strictly increasing (with the order given by the orientation) by (ii) and surjective by (iv). Then $m_v$ is in fact a homeomorphism, because the domain and the codomain of the function are topologically the real line. For $s>0$, we set $h_s(v)=m_v^{-1}(s)$, so that the measure of the interval $(v,h_s(v))$ is $s$. There is a similar map on the negative half horocycle $H^u_-(v)$ that allows us to define the flow $h_s(v)$ for negative time, and we also set the obvious relation $h_0(v)=v$.

It is clear that the flow satisfies the group law, $h_{s_1}\circ h_{s_2}=h_{s_1+s_2}$, because of the additivity of the measure and property (i). For the same reasons, the measure of every interval $I\subset H^u(v)$ (hence, every measurable set) is preserved, $\mu_{H^u(v)}(h_s(I))=\mu_{H^u(v)}(I)$. Thanks to the product structure of the measure (Equation \ref{product_structure}), we deduce that $h_s$ preserves $\mu$. The expanding property of the measures on horocycles is transformed into the commutation relation $g_t\circ h_s = h_{se^{\delta t}} \circ g_t $ between the geodesic flow and the horocyclic flow. Only the continuity of $h_s$ remains to be proved.

\begin{lemma}
	The flow
	\begin{equation*}
		\begin{matrix}
		\RR\times \tilde{\Sigma} & \longrightarrow & \tilde{\Sigma}  \\
		(s,v) & \longmapsto & h_s(v)
		\end{matrix}
	\end{equation*}
	is continuous.
\end{lemma}

\begin{proof}
	Let $s\in \RR$ and $v\in \tilde{\Sigma}$ and consider sequences $s_k\to s$ and $v_k \to v$. We know that the horocycles $H^u(w)$ depend continuously on $w$, so for each $k$ there exists a vector $w_k\in H^u(v_k)$ such that the sequence $\{w_k\}_k$ converges to $h_s(v)$. By Lemma \ref{continuity_measure}, we have $\mu_{H^u(v_k)}((v_k,w_k))\to \mu_{H^u(v)}((v,h_s(v)))=\modul{s}$. We deduce then that the measures of the intervals $(w_k,h_{s_k}(v_k))$ tend to $0$. If the distance between $w_k$ and $h_{s_k}(v_k)$ tends to $0$ too, then we obtain $h_{s_k}(v_k)\to h_s(v)$ so the flow is continuous at $(s,v)$.
	
	Otherwise, we get a contradiction. To see this, suppose that, for some $\varepsilon>0$ and subsequence $k_i$, the Riemannian distance $d_1(w_{k_i},h_{s_{k_i}}(v_{k_i}))$ is greater than $\varepsilon$. Then, since $w_{k_i}\to h_s(v)$, for $i$ big enough $h_{s_{k_i}}(v_{k_i})$ is at distance greater than $\varepsilon/2$ from $h_s(v)$. But the sequence $h_{s_{k_i}}(v_{k_i})$ must accumulate at some point $\zeta$ in $H^u(v)\cup \{v_-\}$, outside of a ball centered at $h_s(v)$. Again by the continuity of the measure, it follows that $\mu_{H^u(v)}((h_s(v),\zeta))=0$, which is impossible because the interval is nonempty.
\end{proof}

\subsection{Unique ergodicity of the horocyclic flow on $\Sigma$}

To study the ergodic properties of the horocyclic flow we introduce the Birkhoff averages. Let $f:T^1M \rightarrow \RR$ be a Borel function and $\tilde{f}:T^1\tilde{M}\rightarrow \RR$ its lift. For a number $R>0$ and $v$ in $\tilde{\Sigma}$, we define
\begin{equation*}
	M_R(f)(v):=\frac{1}{R}\int_0^R\tilde{f}(h_s(v))ds.
\end{equation*}
A simple computation using the commutation relation between the geodesic and the horocyclic flow shows that $M_R(f\circ g_t)=M_{Re^{\delta t}}(f)\circ g_t$. 

Moreover, if we suppose that $f$ is bounded and uniformly continuous, the equidistribution under the action of the geodesic flow we showed in Theorem \ref{theorem_timeed} implies that the Birkhoff averages $M_1(f\circ g_t)$ converge pointwise to $\int f \dd \mu/\mu(T^1 M)$ when the time $t$ goes to $+\infty$. However, we need to understand the behavior of $M_R(f)$ when $R$ goes to infinity, that is to say, the equidistribution of horocycles in length. To do this we will use the relation $M_1(f\circ g_t)=M_{e^{\delta t}}(f)\circ g_t$ and some kind of uniform convergence of the averages $M_1(f\circ g_t)$ towards the average of $f$ on the unitary tangent bundle of the manifold $M$, which we are going to prove.

It is clear from the continuity of the measures on horocycles (Lemma \ref{continuity_measure}) that the function $M_1(f\circ g_t)$ is continuous on $\tilde{\Sigma}$. We can prove the following improved result.

\begin{proposition}\label{global_equicont}
	Let $M$ be an orientable rank $1$ complete connected Riemannian surface with nonpositive curvature satisfying the duality condition. Let $f$ be a bounded and uniformly continuous function on $T^1M$. Then the family of functions $\{M_1(f\circ g_t)\}_{t>0} $ is equicontinuous at every vector of $\tilde{\Sigma}$.
\end{proposition}

\begin{proof}
	Let $v$ be a vector in $\tilde{\Sigma}$. The average of the horocyclic flow can be written explicitly as
	\begin{equation*}
		M_1(f\circ g_t)(w)=\int_{(w,h_1(w))}\tilde{f} \circ g_t \,\dd \mu_{H^u(w)} =	
	\end{equation*}
	\begin{equation}\label{equation_key}
=	\int_{(w_+,h_1(w)_+)} \tilde{f} \circ g_t (P_{H^u(w)}^{-1}(\eta))	e^{\delta \beta_\eta(x_0,\,\pi(P_{H^u(w)}^{-1}(\eta)))} \dd \mu_{x_0}(\eta).
	\end{equation}
	
	Fix $\varepsilon >0$. We consider a relatively compact neighborhood $U$ of $v$ such that, for all $w\in U$,
	\begin{equation*}
		\mu_{x_0}((v_+,h_1(v)_+)\triangle (w_+,h_1(w)_+))<\varepsilon.
	\end{equation*}
	Let $C$ be a uniform bound of $\exp({\delta \beta_\eta(x_0,\,P_{H^u(w)}^{-1}(\eta))})$ for $w\in U$ and $\eta$ in a compact neighborhood of $\overline{(v_+,h_1(v)_+)}$. When $w$ approaches $v$, the set $(w_+,h_1(w)_+)$ will be contained in this compact neighborhood. Then we can change the domain of integration in Equation \ref{equation_key} to $(v_+,h_1(v)_+)$ with an error of $\varepsilon \norm{f}_\infty C$ at most.
	
	By the uniform continuity of $\tilde{f}$, there is a number $r>0$ such that $|\tilde{f}(w)-\tilde{f}(w')|<\varepsilon $ if $d_1(w,w')<r$. If $w$ is close enough to $v$, for all $\eta \in (v_+,h_1(v)_+)$, $P_{H^u(w)}^{-1}  (\eta)$ is at distance less than $r$ from $ P_{H^u(v)}^{-1}(\eta)$. Applying the geodesic flow $g_t$, $t\ge 0$ to these two vectors, their distance does not increase. Hence, when $w$ is close to $v$,
	\begin{equation*}
		\forall t\ge 0,\,\forall \eta\in (v_+,h_1(v)_+), \quad |\tilde{f}(g_t(P_{H^u(w)}^{-1} (\eta)))-\tilde{f}(g_t(P_{H^u(v)}^{-1} (\eta)))|<\varepsilon .
	\end{equation*}
	This is essentially the same we did in property (ii) of the proof of Proposition \ref{prop_local_ed}. We can also control the difference between $\exp({\delta \beta_\eta(x_0,\,P_{H^u(w)}^{-1}(\eta))})$ and $\exp({\delta \beta_\eta(x_0,\,P_{H^u(v)}^{-1}(\eta))})$ if we consider a $w$ close enough to $v$. So the values of the functions at $w$ are close to the values at $v$ uniformly in $t$ when the two vectors are close. This shows that $\{M_1(f\circ g_t)\}_{t>0} $ is equicontinuous at $v$.
\end{proof}

We also observe that the function $M_1(f\circ g_t)$ is bounded by the uniform norm $\norm{f}_\infty$. Both the set $\tilde{\Sigma}$ and the functions $M_R(f):\tilde{\Sigma} \rightarrow \RR$ are invariant under $\Gamma$, so they descend respectively to a set $\Sigma\subset T^1 M$ and some functions $\bar{M}_R(f): \Sigma \rightarrow \RR$. We will apply the Arzelà-Ascoli theorem on the space of continuous functions $C(K)$ over a compact set $K\subset \Sigma$. For every uniformly continuous and bounded function $f:T^1 M\rightarrow \RR$, the family $\{\bar{M}_1(f\circ g_t)|_K\}_{t>0} \subset C(K)$ is equicontinuous and uniformly bounded, so it is a relatively compact subset of $C(K)$ in the uniform convergence topology. This is enough to prove that $\bar{M}_1(f\circ g_t)|_K$ converges uniformly to $\int f \dd \mu/\mu(T^1M)$ when $t\to +\infty$, or we get a contradiction otherwise. Indeed, if we suppose that there is not uniform convergence, there will exist a constant $\varepsilon >0$, a sequence $t_n\to +\infty$ and points $w_n$ in $K$ such that
\begin{equation}\label{contradiction}
	\left|\bar{M}_1(f\circ g_{t_n})(w_n)-\frac{\int f \dd \mu}{\mu(T^1M)} \right|\ge \varepsilon.
\end{equation}
However, because of the sequential compactness of the closure of $\{\bar{M}_1(f\circ g_t)|_K\}_{t>0}$, there exists a subsequence $t_{n_k}$ where $\bar{M}_1(f\circ g_{t_{n_k}})|_K$ converges uniformly to some function $\varphi$ in $C(K)$. In particular, $\bar{M}_1(f\circ g_{t_{n_k}})|_K$ converges pointwise to $\varphi$, but also to $\int f \dd \mu/\mu(T^1M)$ as mentioned above (consequence of Theorem \ref{theorem_timeed}), so $\varphi$ is constant and equal to $\int f \dd \mu/\mu(T^1M)$. But $\bar{M}_1(f\circ g_{t_{n_k}})$ converging to $\int f \dd \mu/\mu(T^1M)$ in the uniform norm contradicts Equation \ref{contradiction}. We have shown:

\begin{proposition}\label{uniform_convergence}
	Let $M$ be an orientable rank $1$ complete connected Riemannian surface with nonpositive curvature satisfying the duality condition. Assume that the Bowen-Margulis measure $\mu$ is finite. Let $f$ be a bounded and uniformly continuous function on $T^1M$. Then the functions $\bar{M}_1(f\circ g_{t}):\Sigma\rightarrow \RR$ converge uniformly on compact sets to the constant $\int f \dd \mu/\mu(T^1M)$ when the time $t$ tends to $+\infty$.
\end{proposition}

Finally, we state Theorem \ref{unique_ergodi} again and prove it with the help of the Birkhoff averages. Recall that $\Sigma$ is the set of vectors whose horosphere contains a rank $1$ recurrent vector, and $\Sigma$ has full $\mu$-measure in $T^1 M$.

\begin{theorem}
	Let $M$ be an orientable rank $1$ complete connected Riemannian surface with nonpositive curvature satisfying the duality condition. Assume that the Bowen-Margulis measure $\mu$ is finite. Then every finite Borel measure on $\Sigma$ invariant under the horocyclic flow $h_s$ is a constant multiple of the Bowen-Margulis measure $\mu|_\Sigma$ restricted to $\Sigma$.
\end{theorem}

\begin{proof}
	Firstly, let us prove that, for every bounded and uniformly continuous function $f:T^1 M\rightarrow \RR$ and for every vector $v$ in $\Sigma$, there exists a sequence $t_n\to +\infty$ such that the Birkhoff integral $\bar{M}_{e^{t_n}}(f)(v)$ tends to $\lambda:=\int f \dd \mu/\mu(T^1M)$.
	
	There is a recurrent vector $w$ in the unstable horocycle of $v$, since $v$ is in $\Sigma$. Let $t_n$ be a sequence tending to $+\infty$ such that $g_{-t_n}(w)\to w$. Then obviously $g_{-t_n}(v)$ also tends to $w$. We consider the compact set $K:=\{g_{-t_n}(v)\}_{n\ge 0}\cup \{w\}\subset \Sigma$. By Proposition \ref{uniform_convergence}, the functions $\bar{M}_1(f\circ g_{t})$ converge uniformly on $K$ to the global average $\lambda$ of $f$. Therefore, using the time-scale relation, we have
	\begin{equation*}
		|\bar{M}_{e^{t_n}}(f)(v)-\lambda|=|\bar{M}_1(f\circ g_{t_n})(g_{-t_n}(v))-\lambda|\le
		\sup_{u\in K}|\bar{M}_1(f\circ g_{t_n})(u)-\lambda|\xrightarrow{n\to +\infty}0.
	\end{equation*}
	
	We can now prove that the restriction $\mu|_\Sigma$ of $\mu$ to $\Sigma$ is the unique measure on $\Sigma$ invariant under $h_s$, up to a multiplicative constant. Suppose that $\nu$ is an ergodic $h_s$-invariant probability measure on $\Sigma$. By the Birkhoff ergodic theorem, for every bounded and uniformly continuous function $f:\Sigma\rightarrow\RR$, for $\nu$-a.e. $v$ in $\Sigma$, we have
	\begin{equation*}
		\bar{M}_R(f)(v)=\frac{1}{R}\int_0^R f(h_s(v))\dd s\xrightarrow{R\to +\infty}\int_{\Sigma} f \dd \nu.
	\end{equation*}
	We take $v$ one of the points of $\Sigma$ where $\bar{M}_R(f)$ converges to $\int f \dd \nu$. We can extend $f$ to a bounded and uniformly continuous function $\hat{f}$ on $T^1 M$, because $\Sigma$ is dense in $T^1 M$. As we have seen, there is a sequence $R_n=e^{t_n}$ where $\bar{M}_R(\hat{f})(v)=\bar{M}_R(f)(v)$ tends to $\lambda$ as well. So we obtain
	\begin{equation*}
		\int_{\Sigma} f \dd \nu =\lambda =\frac{\int_{T^1M} f\dd \mu}{\mu(T^1 M)}
		=\frac{\int_{\Sigma} f\dd \mu}{\mu(\Sigma)},
	\end{equation*}
	because $\Sigma$ has full $\mu$-measure. We have concluded that $\nu$ is equal to the normalization of $\mu|_\Sigma$.
\end{proof}

\subsection{Alternative proof of the unique ergodicity}

We would like to point out another way to prove Proposition \ref{uniform_convergence}, which does not require the equidistribution of horocycles (Theorem \ref{theorem_timeed}). Instead, we use a version of the Arzelà-Ascoli theorem for the compact-open topology, the ergodic theorem and the fact that there exists a dense horocycle in $\Sigma$. Actually, we can prove that all the horocycles of $\Sigma$ are dense.

\begin{lemma}\label{dense_orbit}
	Let $M$ be a rank $1$ nonpositively curved complete connected Riemannian surface with the duality condition. Then every horocycle $H$ contained in $\Sigma$ is dense in $T^1M$.
\end{lemma}
\begin{proof}
	It follows directly from two results of Eberlein. A nonpositively curved complete connected manifold that satisfies the visibility axiom and the duality condition, like $M$, has a dense horocycle in $T^1M$ \cite[Theorem 5.2]{Eberlein73}. Next we apply \cite[Theorem 5.5]{Eberlein73} to $M$, which says that a horocycle $H^u(v)$ in $\mathcal{H}$ is dense in $T^1 M$ if and only if $v$ is not almost minimizing. We say that $v$ is \emph{almost minimizing} if there exists a constant $C>0$ such that, for all $t\ge0$, we have $d(\pi(v),\pi(g_t(v)))\ge t -C$. If a horocycle $H$ is contained in $\Sigma$, then there is a recurrent vector in $H$ and, in particular, this vector is not almost minimizing. Thus, the horocycle $H$ is dense in $T^1 M$
\end{proof}

We consider the space of continuous functions $C(\Sigma)$ on the set $\Sigma$ equipped, this time, with the compact-open topology. Recall that, for functions on a metric space, the convergence in the compact-open topology is equivalent to the uniform convergence on compact subsets. The next fact will be used in the alternative proof of Proposition \ref{uniform_convergence}. 

\begin{lemma}\label{coimpliesl2}
	If a uniformly bounded sequence of functions $f_n$ in $C(\Sigma)$ converges in the compact-open topology to a function $f\in C(\Sigma)$, then the sequence also converges to $f$ in the $L^2(\Sigma,\mu|_\Sigma) $-norm.
\end{lemma}
\begin{proof}
	Recall that a Polish space is a separable completely metrizable topological space. The set of rank $1$ recurrent vectors $Rec^1\subset\Sigma$ is a $G_\delta $ subset of $T^1 M$. Then $Rec^1$ is a Polish space, because $G_\delta$ subsets of Polish spaces are Polish spaces and $T^1 M$ is Polish.
	
	Finite Borel measures on a Polish space are Radon. For $Rec^1 $ and the restriction of $\mu$ to $Rec^1$, this means that, given any number $\varepsilon>0$, there is a compact set $K\subset Rec^1$ such that $\mu(Rec^1 \setminus K)<\varepsilon$. In addition, the set $Rec^1 $ is of full measure in $T^1 M$. We have 
	
	\begin{equation*}
		\int_\Sigma \modul{f_k-f }^2 \,\dd \mu \le 4 C^2 \varepsilon   + \int_K \modul{f_k-f }^2 \,\dd \mu,
	\end{equation*}
	where $C$ is a bound of $f$ and the sequence $\{f_n\}$. The last term tends to $0$ because of the uniform convergence of $f_n$ to $f$ on compact subsets.
\end{proof}

Let $f$ be a bounded and uniformly continuous function on $T^1M$. Applying the Arzelà-Ascoli theorem for the compact-open topology \cite[Theorem XII.6.4]{Dugundji66}, since the family of functions $\{\bar{M}_1(f\circ g_t)\}_{t>0}$ is equicontinuous and uniformly bounded, we obtain that it has a compact closure in $C(\Sigma)$ with the compact-open topology. To complete the proof of Proposition \ref{uniform_convergence}, we show that the only accumulation point of $\{\bar{M}_1(f\circ g_t)\}_{t>0}$ in the compact-open topology is the constant function $\int f\dd\mu/\mu(T^1M)$.

Let $\varphi$ in $C(\Sigma)$ be the limit of a sequence $\bar{M}_1(f\circ g_{t_k})$ in the compact-open topology, where $t_k\to +\infty$. By Lemma \ref{coimpliesl2}, $\varphi$ is the limit in $L^2(\Sigma,\mu|_\Sigma)$ of the same sequence. On the other hand, we apply the $L^2$ ergodic theorem for the system $(\Sigma,h_s,\mu|_\Sigma)$ to the function $f\in L^2(\Sigma,\mu|_\Sigma)$. It says that $\bar{M}_t(f)$ converges to an $h_s$-invariant function $\bar{f}$ in the $L^2$ norm, with the equality $\int \bar{f}\dd\mu=\int f\dd\mu$. Thanks to $g_t$ invariance of $\mu$, we have the inequality

\begin{equation*}
	\norm{\varphi - \bar{f}\circ g_{t_k}}_2\le
	\norm{\varphi - \bar{M}_1(f\circ g_{t_k})}_2 + \norm{\bar{f} - \bar{M}_{e^{t_k}}(f)}_2,
\end{equation*}
which implies the $L^2$-convergence of $\bar{f}\circ g_{t_k}$ to $\varphi$, because both terms on the right side tend to zero. The function $\bar{f}$ is $h_s$-invariant, so $\bar{f}\circ g_{t_k}$ are also invariant because the geodesic and the horocyclic flow commute. Then their limit, the continuous function $\varphi$ of $\Sigma$, is also invariant under $h_s$.

In brief, the function $\varphi$ is constant on the orbits of $h_s$, and these orbits are dense by Lemma \ref{dense_orbit}. Since $\varphi $ is continuous, we conclude that it is constant on $\Sigma$. In fact, the value of the constant is $\int f\dd\mu/\mu(T^1M)$, because we have

\begin{equation*}
	\int \varphi \, \dd \mu = \int \bar{f} \circ g_{t_k} \, \dd \mu =\int \bar{f} \,\dd \mu = \int f\, \dd \mu.
\end{equation*}

\subsection{Final remark}

Theorem \ref{unique_ergodi} does not solve completely the problem of the horocyclic flow in nonpositive curvature, since it does not say what happens to the flow outside the set $\Sigma$. For instance, we wonder if a horocyclic flow defined everywhere on a compact nonpositively curved surface is uniquely ergodic. We will study this question for the class of compact manifolds without flat strips in a forthcoming article.

\bibliographystyle{plain}
\bibliography{general_library}
\end{document}